\documentclass[12pt,a4paper]{article}
\usepackage{latexsym,amsfonts,amsmath,graphics,amsthm}
\usepackage{epsfig}
\usepackage{mathrsfs}
\usepackage{color}

\newtheorem{theorem}{Theorem}
\newtheorem{lemma}{Lemma}
\newtheorem{corollary}{Corollary}
\newtheorem{proposition}{Proposition}

\newcommand{\norm}[1]{\left\Vert#1\right\Vert}
\newcommand{\abs}[1]{\left\vert#1\right\vert}

\newcommand{\To}{\rightarrow}

\newcommand{\satop}[2]{\stackrel{\scriptstyle{#1}}{\scriptstyle{#2}}}

\newcommand{\bsalpha}{\boldsymbol{\alpha}}

\newcommand{\bsgamma}{\boldsymbol{\gamma}}

\newcommand{\bsk}{\boldsymbol{k}}

\newcommand{\bsx}{\boldsymbol{x}}

\newcommand{\bsg}{\boldsymbol{g}}

\newcommand{\bst}{\boldsymbol{t}}
\newcommand{\bsp}{\boldsymbol{p}}

\newcommand{\bsw}{\boldsymbol{w}}

\newcommand{\LL}{{\cal L}}

\newcommand{\bsy}{\boldsymbol{y}}
\newcommand{\bssigma}{\boldsymbol{\sigma}}

\newcommand{\wal}{{\rm wal}}

\newcommand{\icomp}{\mathrm{i}}
\newcommand{\bszero}{\boldsymbol{0}}
\newcommand{\bsone}{\boldsymbol{1}}
\newcommand{\rd}{\,\mathrm{d}}

\newcommand{\NN}{\mathbb{N}}
\newcommand{\ZZ}{\mathbb{Z}}

\newcommand{\EE}{\mathbb{E}}

\newcommand{\FF}{\mathbb{Z}}
\newcommand{\QQ}{\mathbb{Q}}

\newcommand{\RR}{{\mathbb R}}
\newcommand{\CC}{{\mathbb C}}
\newcommand{\cH}{{\mathscr H}}
\newcommand{\cHal}{{\mathcal H}}
\newcommand{\cP}{{\mathcal P}}
\newcommand{\uu}{\mathfrak{u}}
\newcommand{\sh}{\mathrm{sh}}

\makeatletter
\newcommand{\rdots}{\mathinner{\mkern1mu\lower-1\p@\vbox{\kern7\p@\hbox{.}}
\mkern2mu \raise4\p@\hbox{.}\mkern2mu\raise7\p@\hbox{.}\mkern1mu}}
\makeatother
\allowdisplaybreaks

\begin{document}

\title{Open type quasi-Monte Carlo integration based on Halton sequences in weighted Sobolev spaces}

\author{Peter Hellekalek, Peter Kritzer and Friedrich Pillichshammer\thanks{The authors are supported by the
Austrian Science Fund (FWF): Projects F5504-N26 (Hellekalek), F5506-N26 (Kritzer), and F5509-N26 (Pillichshammer),
respectively. All these projects are part of the Special Research Program "Quasi-Monte Carlo Methods: Theory and Applications".}}

\date{}

\maketitle

\begin{abstract}
\noindent In this paper, we study quasi-Monte Carlo (QMC) integration in
weighted Sobolev spaces. In contrast to many previous results
the QMC algorithms considered here are of open type, i.e., they are extensible in
the number of sample points without having to discard the
samples already used. As the underlying integration nodes we consider randomized
Halton sequences in prime bases $\bsp=(p_1,\ldots,p_s)$ for which
we study the root mean square (RMS) worst-case error. The
randomization method is a $\bsp$-adic shift which is based on
$\bsp$-adic arithmetic. 

The obtained error bounds are optimal in the order of magnitude of
the number of sample nodes. Furthermore we obtain
conditions on the coordinate weights under which the error bounds are
independent of the dimension $s$. In terms of the field of Information-Based
Complexity this means that the corresponding QMC rule achieves a
strong polynomial tractability error bound. Our findings on
the RMS worst-case error of randomized Halton sequences can be
carried over to the RMS $L_2$-discrepancy.

Except for the $\bsp$-adic shift our results are fully
constructive and no search algorithms (such as the
component-by-component algorithm) are required.
\end{abstract}



%%%%%%%%%%%%%%%%%%%%%%%%%%%%%%%%%%%%%%%%%%%%%
%%%%%%%%%%%%%%%%%%%%%%%%%%%%%%%%%%%%%%%%%%%%%

\section{Introduction}\label{secintro}

Point sequences with good distribution properties in a given
domain, as for example the $s$-dimensional unit cube $[0,1)^s$
which is the domain considered in this paper, are of interest in
various branches of mathematics. For instance, in the field of
quasi-Monte Carlo (QMC) integration one requires very well
distributed points in the integration domain as the underlying
nodes for quadrature rules. A QMC rule is an equal-weight
quadrature rule of the form $\frac{1}{N} \sum_{n=0}^{N-1}
f(\bsx_n)$ with points $\bsx_0,\bsx_1,\ldots,\bsx_{N-1} \in
[0,1)^s$ which is used to approximate the integral of a function $f$ over the
$s$-dimensional unit-cube, i.e., $$\int_{[0,1]^s} f(\bsx) \rd \bsx
\approx \frac{1}{N} \sum_{n=0}^{N-1} f(\bsx_n).$$

Often a distinction is made between QMC rules of ``open type'' and
of ``closed type'' (see \cite{DKS13}):
\begin{itemize}
\item An open type QMC rule uses the first $N$ points of an infinite
sequence $\mathcal{S}=(\bsx_n)_{n \ge 0}$ as integration nodes.
Thus to increase $N$, one only needs to evaluate the integrand at the additional
cubature points, and there is no need to discard previous function evaluations.

\item A closed type QMC rule uses a finite point set
$\cP=\{\bsx_0,\bsx_1,\ldots,\bsx_{N-1}\}$, the form of which depends on $N$, as integration nodes.
Thus increasing $N$ usually means that all (or at least some) previous function evaluations are discarded and a
different (larger) point set needs to be used.
\end{itemize}

QMC rules of closed type have been very well studied in the literature, where
common point sets in use are lattice point sets in the sense of
Hlawka~\cite{H62} and Korobov~\cite{K59} or digital nets with
their sub-class of polynomial lattice point sets according to
Niederreiter \cite{N92}. For an overview see, for example,
\cite{DKS13,DP10,LP14,N92,SJ94}.

In this paper we consider QMC rules of open type where we focus on
a special kind of point sequences underlying the QMC rule, namely
Halton sequences (cf.~\cite{H60}). Halton sequences (and their
one-dimensional versions, van-der-Corput sequences), are among the
prototypes of sequences with excellent distribution properties. The definition
of these sequences is based on the radical inverse function. Here
and in the following, we denote by $\NN_0$ the set of nonnegative
integers and by $\NN$ the set of positive integers. The radical
inverse function is defined as follows for an arbitrary integer
$p\ge 2$. For $n\in\NN_0$, let $n=n_0 + n_1 p + n_2 p^2+\cdots$ be
the base $p$ expansion of $n$ (which is of course finite) with
digits $n_i\in\{0,1,\ldots,p-1\}$ for $i\ge 0$. The radical
inverse function $\phi_p:\NN_0\To [0,1)$ in base $p$ is defined by
\[\phi_p (n):=\sum_{k=1}^{\infty} \frac{n_{k-1}}{p^k}.\]
The radical inverse function in base $p$ gives rise to the
definition of Halton sequences. Halton sequences can be defined
for any dimension $s\in\NN$, and for their definition we need $s$
integers. To be more precise, let $p_1,\ldots,p_s$ be $s$
integers, $p_j\ge 2$ for all $j\in [s]:=\{1,\ldots,s\}$, and let
$\bsp=(p_1,\ldots,p_s)$. Then the $s$-dimensional Halton sequence
$\cHal_{\bsp}$ in bases $p_1,\ldots,p_s$ is defined to be the
sequence $\cHal_{\bsp}=(\bsx_n)_{n\ge 0}\subseteq [0,1)^s$, where
\[\bsx_n=(\phi_{p_1}(n),\phi_{p_2}(n),\ldots,\phi_{p_s}(n)),\ n\ge 0.\]
It is well known (see, e.g.,~\cite{DP10,N92}) that Halton
sequences have good distribution properties if and only if the
bases $p_1,\ldots,p_s$ are mutually relatively prime, and for the
sake of simplicity we assume throughout the rest of the paper that
$\bsp=(p_1,\ldots,p_s)$ consists of $s$ mutually different prime
numbers.

In this paper we study integration of functions from a weighted
anchored Sobolev space consisting of functions whose mixed partial
derivatives of order up to one are square integrable. Integration
in this particular Sobolev space has already been studied by many
authors, see, for
instance,~\cite{DP10,KP11,K03,NW08,NW10,NW12,SW98,SW01}. The Sobolev
space can be introduced via a reproducing kernel and will be
presented in Section~\ref{secsobspace}. For a general reproducing
kernel Hilbert space $\cH(K)$ of functions on $[0,1]^s$ with
reproducing kernel $K$ and norm $\| \cdot \|_K$ the worst-case
error of a QMC rule based on the first $N$ points of a sequence $\mathcal{S}=(\bsx_n)_{n
\ge 0}$ is defined as $$e_{N,s}(\mathcal{S},K)=\sup_{\|f\|_K \le
1} \left|\int_{[0,1]^s} f(\bsx) \rd \bsx
-\frac{1}{N}\sum_{n=0}^{N-1} f(\bsx_n)\right|,$$ where the
supremum is extended over all functions $f$ which belong to the
unit ball of $\cH(K)$. It has been shown in
\cite[Proposition~1]{HKKN} that if
\begin{equation}\label{condHickKri}
\inf_{\bst \in [0,1]^s} \sup_{\|f\|_K \le 1} \left|\int_{[0,1]^s}
f(\bsx)\rd \bsx -f(\bst)\right| \ge c_K >0
\end{equation}
for some absolute constant  $c_K$ (which may depend on the kernel function $K$), then for any open type QMC
rule based on a sequence $\mathcal{S}$, the sequence of worst-case errors
$(e_{N,s}(\mathcal{S},K))_{N \ge 1}$ cannot decrease to zero faster than $O(N^{-1})$.\\

The tool to analyze integration by means of Halton sequences in
our Sobolev space is to introduce an orthonormal basis
$B_{\bsp}^s$, where $\bsp=(p_1,\ldots,p_s)$ consists of $s$
mutually different primes, of $L_2 ([0,1]^s)$ that matches the
structural properties of Halton sequences in bases
$p_1,\ldots,p_s$. The function system $B_{\bsp}^s$ is based on
arithmetic over fields of $p_j$-adic ($j\in [s]$) integers and is
referred to as $\bsp$-adic function system; the $\bsp$-adic
function system, which will be introduced in detail in Section~\ref{secprel},
was first studied in~\cite{H09}. In line with the properties of
$B_{\bsp}^s$, we will introduce a way of randomizing Halton
sequences by means of a random $\bsp$-adic shift, which preserves
the structural properties of the sequences. Introducing a random
element in a QMC integration node set facilitates the error
analysis in a Sobolev space as the one considered here. This has
been done previously by using lattice point sets randomized by a
shift modulo one and the trigonometric function system (see,
e.g.,~\cite{K03}), and also by using polynomial lattice points randomized by a digital shift and
the Walsh function system (see, e.g.,~\cite{DP05}). In this
context, we are then able to study the integration error for the
Sobolev space in the sense of a root mean square (RMS) error,
where the mean is considered with respect to the randomization
method.

We show that the RMS worst-case error is of optimal order in the
number $N$ of employed sample nodes for all $N \in \NN$, and we also
study the dependence of the error bounds on the dimension
$s$. We give conditions on the coordinate weights which guarantee that these
error bounds are independent on the dimension $s$. The main result
will be presented in Theorem~\ref{thmerrsob} in
Section~\ref{secsobspace}.

We stress that the present paper is the first paper to address
integration in a Sobolev space by using $\bsp$-adically shifted
Halton sequences and the $\bsp$-adic function system. We furthermore emphasize
that, apart from the randomization, in our approach we explicitly
know the underlying point set, as opposed to the case of lattice
points and polynomial lattice points, which need to be found by a
component-by-component algorithm (see again~\cite{DP05,K03}).
Hence, from the point of view of providing explicitness, our
approach is preferable over using (polynomial) lattice points. The
price we have to pay for this advantage is that so far we have
stricter conditions regarding tractability results (see
Section~\ref{secsobspace} for details on this matter), but we suspect
that this is a purely technical problem.

The rest of the paper is structured as follows. In
Section~\ref{secprel}, we define the $\bsp$-adic function system
$B_{\bsp}^s$. In Section~\ref{secsobspace} we introduce the
Sobolev space $\cH (K_{s,\bsgamma})$ and present the main results
on integration using $\bsp$-adically shifted Halton sequences.
These results will then also yield results on the RMS
$L_2$-discrepancy of $\bsp$-adically shifted Halton sequences which
will be presented in Section~\ref{secRMSL2}. In
Section~\ref{final} we give some comments on QMC integration
in other, but related, reproducing kernel Hilbert spaces.

\paragraph{Notation.} Throughout the paper we use the following notation.
For functions $f,g:\NN \rightarrow \RR^+$ we write $g(N) \ll f(N)$
(or $g(N) \gg f(N)$), if there exists a $C>0$ such that $g(N) \le
C f(N)$ (or $g(N) \ge C f(N)$) for all $N \in \NN$, $N \ge 2$. If
we would like to stress that the quantity $C$ may also depend on
other variables than $N$, say $\alpha_1,\ldots,\alpha_w$, this
will be indicated by writing $\ll_{\alpha_1,\ldots,\alpha_w}$ (or
$\gg_{\alpha_1,\ldots,\alpha_w}$). Furthermore, $[s]$ denotes the
set of coordinate indices, i.e., $[s]=\{1,2,\ldots,s\}$ and $\log$
denotes the natural logarithm. Throughout the paper we use the
notation $\bsp=(p_1,\ldots,p_s)$ to denote a vector of $s$
mutually different prime numbers.

\section{The $\bsp$-adic function system}\label{secprel}

In this section, we define a function system that is based on
$p$-adic (or, in the $s$-dimensional case, $p_1$-, $p_2$-, \ldots,
$p_s$-adic) integer arithmetic and that forms an orthonormal basis
of $L_2([0,1]^s)$. For the sake of simplicity, we restrict
ourselves to the case where $p$ or $p_1,\ldots,p_s$, respectively,
are prime numbers, in the following. However, we remark that large
parts of the theory explained below also work if we simply assume
that the bases $p_1,\ldots,p_s$ are pairwise co-prime.

Let $p$ be a prime number. We define the set of $p$-adic numbers
as the set of formal sums
\begin{equation*}
\mathbb{Z}_p = \left\{z = \sum_{r=0}^\infty z_r p^r\, : \, z_r \in
\{0,1,\ldots,p-1\} \mbox{ for all } r \in \NN_0\right\}.
\end{equation*}
The set of nonnegative integers $\mathbb{N}_0$ is a subset of
$\mathbb{Z}_p$. For two nonnegative integers $y, z \in
\mathbb{N}_0 \subseteq \mathbb{Z}_p$, the sum $y+z \in
\mathbb{Z}_p$ is defined as the usual sum of integers. The
addition can be extended to all $p$-adic numbers.
The set $\mathbb{Z}_p$ with this addition then forms an abelian
group. For instance,
the inverse of $1 \in \mathbb{Z}_p$ is given by the formal sum
\begin{equation*}
(p-1) + (p-1) p + (p-1) p^2 + \cdots.
\end{equation*}
We then have
\begin{align*}
1 + [(p-1) + (p-1) p + (p-1) p^2 + \cdots ]  = & 0 + (1 + (p-1)) p
+ (p-1) b^2 + \cdots \\  = & 0 + 0 p + (1 + (p-1)) p^2 + \cdots \\
= & 0 p + 0 p^2 + \cdots = 0.
\end{align*}

As an extension of the radical inverse function defined in
Section~\ref{secintro}, we define the so-called Monna map
\begin{equation*}
\phi_p:\mathbb{Z}_p \to [0,1)\ \ \mbox{by}\ \ \phi_p(z): =
\sum_{r=0}^\infty \frac{z_r}{p^{r+1}} \pmod{1}
\end{equation*}
whose restriction to $\NN_0$ is exactly the radical inverse
function in base $p$. We also define the inverse
\begin{equation*}
\phi_p^+: [0,1)\to \mathbb{Z}_p\ \ \mbox{by}\ \
\phi_p^+\left(\sum_{r=0}^\infty \frac{x_r}{p^{r+1}}\right) :=
\sum_{r=0}^\infty x_r p^r,
\end{equation*}
where we always use the finite $p$-adic representation for
$p$-adic rationals in $[0,1)$. In this context we say that $x\in
[0,1)$ is a $p$-adic rational if it can be represented by a finite
$p$-adic representation. More precisely, the $p$-adic
rationals in $[0,1)$ are the set $\QQ(p^{\infty}) =\bigcup_{r \ge
0} \QQ(p^r)$, where $\QQ(p^r)=\{m p^{-r}\ :\ m \in
\{0,1,\ldots,p^r-1\}\}$. Elements of $[0,1)\setminus
\QQ(p^{\infty})$ are called $p$-adic irrationals. Note that the
$p$-adic rationals are a set of measure zero.

For $k \in \mathbb{N}_0$ we can define characters of
$\mathbb{Z}_p$ by
\begin{equation*}
\,_p\chi_k:\mathbb{Z}_p \to \{c \in \mathbb{C}: |c| = 1\},\ \
\mbox{ where }\ \ \,_p\chi_k(z) = \mathrm{e}^{2\pi \mathrm{i}
\phi_p(k) z},
\end{equation*}
where $\mathrm{i}=\sqrt{-1}$. It is easily checked that these
functions satisfy $\,_p\chi_k(y+z) = \,_p\chi_k(y) \,_p\chi_k(z)$,
$\,_p\chi_k(0) = 1$, $\,_p\chi_0(z) = 1$, $\,_p\chi_k(z)
\,_p\chi_l(z) = \,_p\chi_{\phi_p^+(\phi_p(k) + \phi_p(l) \pmod{1})}(z)$.

Furthermore, we define
\begin{equation*}
\,_p\beta_k: [0,1) \to \{c \in \mathbb{C}: |c| = 1\},\ \
\mbox{where}\ \ \,_p\beta_k(x) = \,_p\chi_k(\phi_p^+(x)).
\end{equation*}
It is easy to verify that we have $\,_p\beta_k(x) \,_p\beta_l(x) =
\,_p\beta_{\phi_p^+(\phi_p(k) + \phi_p(l)\pmod{1})}(x)$ and
$\overline{\,_p\beta_k(x)}=\,_p\beta_{\phi_p^+(-\phi_p(k)\pmod{1})}(x)$,
where $\overline{z}$ denotes the complex conjugate of $z \in \CC$.

For $\bsp=(p_1,\ldots,p_s)$ the $s$-dimensional version
$\,_{\bsp}\beta$ is defined as follows: for $\bsk=(k_1,\ldots,k_s)
\in \NN_0^s$ and for $\bsx=(x_1,\ldots,x_s) \in [0,1)^s$ define
$$\,_{\bsp}\beta_{\bsk}(\bsx):=\prod_{j=1}^s \,_{p_j}\beta_{k_j}(x_j).$$
From now on we suppress the dependence of $\chi$ and $\beta$ on
$p$ and $\bsp$, respectively and we simply write $\chi_k$,
$\beta_k$ and $\beta_{\bsk}$, respectively. The choice of $p$ and
$\bsp$ will be clear from the context.

The following proposition states that the system of the functions
$\beta_{\bsk}$, $\bsk\in\NN_0^s$, is an orthonormal basis of
$L_2([0,1]^s)$. For a proof of this property we refer to
\cite[Corollary~3.10]{hel2010}.
\begin{proposition}[ONB property]\label{ONB_prop}
The system $B_{\bsp}^{(s)}:=\{\beta_{\bsk}\, : \, \bsk \in
\NN_0^s\}$ is an orthonormal basis of $L_2([0,1]^s)$.
\end{proposition}

We refer to \cite{H09,hel2010,helnie} for more information on
$\bsp$-adic functions.

\section{Open type quasi-Monte Carlo rules based on randomized Halton sequences}\label{secsobspace}

In the following we study QMC integration in a weighted anchored
Sobolev space.

\subsection{The weighted anchored Sobolev space}

We study numerical integration in the weighted anchored Sobolev
space $\cH (K_{s,\bsgamma})$ with anchor $\bsone=(1,1,\ldots,1)$
consisting of functions on $[0,1]^s$ whose first mixed partial
derivatives are square integrable (likewise one could consider any
anchor $\bsw \in [0,1]^s$ which would lead to similar results as
the ones we will obtain in the following, see
Section~\ref{final}). This space is a reproducing kernel Hilbert
space (see \cite{A50} for general information on reproducing
kernel Hilbert spaces) with kernel function
\begin{equation}\label{ker_sob_space}
K_{s,\bsgamma}(\bsx,\bsy)=\prod_{j=1}^s (1+\gamma_j \min
(1-x_j,1-y_j)) \ \ \mbox{ for } \bsx,\bsy \in [0,1]^s,
\end{equation}
where $\bsx=(x_1,x_2,\ldots,x_s)$, $\bsy=(y_1,y_2,\ldots,y_s)$ and
where $\bsgamma=(\gamma_j)_{j \ge 1}$ is a sequence of weights
$\gamma_j \in \RR^+$ which model the influence of the coordinate
$j$ on the integrands\footnote{The idea of weighted function
spaces was introduced by Sloan and Wo\'{z}niakowski \cite{SW98},
in order to explain the success of quasi-Monte Carlo algorithms in
practical applications.}. The inner product is given by
\[\langle f,g \rangle_{K_{s,\bsgamma}}=\sum_{\uu\subseteq [s]}\gamma_{\uu}^{-1} \int_{[0,1]^{\abs{\uu}}}
  \frac{\partial^{\abs{\uu}}}{\partial \bsx_{\uu}} f(\bsx_{\uu},\bsone)\frac{\partial^{\abs{\uu}}}{\partial \bsx_{\uu}}
g(\bsx_{\uu},\bsone)\rd \bsx_{\uu}.\] Here
$\gamma_{\uu}=\prod_{j\in\uu}\gamma_j$; in particular
$\gamma_{\emptyset}=1$. Furthermore, we denote by
$f(\bsx_{\uu},\bsone)$ the value of $f$ at the point where all
components $x_j$ of $\bsx\in [0,1]^s$ for which $j\not\in \uu$ are
replaced by 1, and $\frac{\partial^{\abs{\uu}}}{\partial
\bsx_{\uu}} h$ denotes the derivative of a
function $h$ with respect to the $x_j$ with $j\in\uu$. The norm in
$\cH (K_{s,\bsgamma})$ is given by
$\norm{f}_{K_{s,\bsgamma}}=\sqrt{\langle
f,f\rangle_{K_{s,\bsgamma}}}$. The Sobolev space $\cH
(K_{s,\bsgamma})$ has been studied frequently in the literature
(see, among many references,
e.g.~\cite{DKPS05, DP05, KP11, K03, NW08, SW98, wang}).\\

It is well known that the worst-case error for QMC integration in
$\cH (K_{s,\bsgamma})$ is exactly the weighted $L_2$-discrepancy
of the underlying point set or sequence (this was first shown in
\cite{SW98}; see also \cite{DP14a}), i.e.,
\begin{equation}\label{errL2disc}
e_{N,s}(\mathcal{S},K_{s,\bsgamma})=L_{2,N,\bsgamma}
(\mathcal{S}),
\end{equation}
where $L_{2,N,\bsgamma}$ is the weighted $L_2$-discrepancy of the
sequence $\mathcal{S}$, which is given by
\begin{eqnarray*}
L_{2,N,\bsgamma}({\mathcal S}) = \left(\sum_{\emptyset \not= \uu
\subseteq [s]} \bsgamma_{\uu}\int_{[0,1]^{|\uu|}} |
\Delta_N(\bst_{\uu},\bsone)|^2 \rd \bst_{\uu}\right)^{1/2}.
\end{eqnarray*}
Here $\Delta_N$ is the so-called local discrepancy of
$\mathcal{S}$, given by
\begin{equation*}
\Delta_N(\bst) = \frac{1}{N} \sum_{n=1}^N
1_{[\bszero,\bst)}(\bsx_n) - \prod_{i=1}^s t_i,
\end{equation*}
where $\bst = (t_1,\ldots, t_s)\in [0,1]^s$ and $1_{[\bszero,\bst)}$ denotes
the characteristic function of the interval
$[\bszero,\bst)=[0,t_1)\times [0,t_2)\times \ldots \times
[0,t_s)$, i.e., $1_{[\bszero,\bst)}(\bsx)$ equals one if $\bsx$ belongs
to $[\bszero,\bst)$ and zero otherwise.

Hence, all results that will be derived below regarding the worst-case error of integration in
$\cH (K_{s,\bsgamma})$ using Halton sequences carry over to the weighted $L_2$-discrepancy.\\

From \eqref{errL2disc} it follows that
\begin{eqnarray*}
\inf_{\bst \in [0,1]^s} \sup_{\|f\|_{K_{s,\bsgamma}} \le 1} \left|\int_{[0,1]^s} f(\bsx)\rd \bsx -f(\bst)\right|
& = &  \inf_{\bst \in [0,1]^s} L_{2,1,\bsgamma}(\{\bst\})\\
& = &  \inf_{\bst \in [0,1]^s}  \left(\sum_{\emptyset \not= \uu \subseteq [s]}
\gamma_{\uu} (L_{2,1,\bsgamma}(\{\bst_{\uu}\}))^2\right)^{1/2}\\
& \ge & \inf_{0 \le t \le 1} (\gamma_1 L_{2,1,\gamma_1}(\{t\}))^{1/2} \\
& = & \sqrt{\gamma_1/12} >0,
\end{eqnarray*}
where for the second equality we refer to \cite{DP14a} and where the last equality is
easily proved with the help of Warnock's formula for the $L_2$-discrepancy (see, for example, \cite{DP10,DP14a}).
Hence condition  \eqref{condHickKri} is satisfied and this implies that the sequence
$(e_{N,s}(\mathcal{S},K_{s,\bsgamma}))_{N \ge 1}$ cannot converge to zero faster than $O(N^{-1})$.\\

QMC integration in $\cH (K_{s,\bsgamma})$ based on the Halton
sequence $\cHal_{\bsp}$ has already been studied by
Wang~\cite{wang}. He proved that if the weights $\bsgamma$ satisfy
\begin{equation}\label{weightcondwang}
\sum_{j=1}^{\infty} \gamma_j^{1/2} j \log j < \infty,
\end{equation}
then for any $\delta>0$ and for all $N \in \NN$ we have
$$e_{N,s}(\cHal_{\bsp},K_{s,\bsgamma}) \ll_{\delta}
\frac{1}{N^{1-\delta}},$$ where the implied constant depends only
on $\delta>0$, but not on $s$ and $N$ (see
\cite[Theorem~5]{wang}). In the terms of Information Based
Complexity (see \cite{NW08,NW10,NW12}) this means that the
corresponding QMC rule achieves a strong polynomial
tractability error bound and the $\varepsilon$-exponent is equal
to one (which is optimal).

\subsection{$\bsp$-adic shifts and $\bsp$-adic shift invariant kernels}

We introduce a method of randomizing the integration nodes in use which
is referred to as a $\bsp$-adic shift. This special case of
randomization fits perfectly to the definition of the Halton
sequence $\cHal_{\bsp}$. We point out already here that if
$\bsp$-adic shifts and Halton sequences are used in conjunction
with each other they are always meant with respect to the same
bases $\bsp$.

For given bases $\bsp=(p_1,\ldots,p_s)$, we define a $\bsp$-adic
shift as follows. For a point $\bsx=(x_1,\ldots,x_s)\in [0,1)^s$,
and a given $\bssigma=(\sigma_1,\ldots,\sigma_s)\in [0,1)^s$, we
define $\bsx\oplus_{\bsp}\bssigma\in [0,1)^s$ to be
\[\bsx\oplus_{\bsp}\bssigma=(x_1\oplus_{p_1}\sigma_1,\ldots,x_s\oplus_{p_s}\sigma_s),\]
where $x_j\oplus_{p_j}\sigma_j=\phi_{p_j} (\phi_{p_j}^+ (x_j) +
\phi_{p_j}^+ (\sigma_j))$. Note that here $\phi_{p_j}^+ (x_j) +
\phi_{p_j}^+ (\sigma_j)$ means addition in $\ZZ_{p_j}$. For a
Halton sequence $\cHal_{\bsp}=(\bsx_n)_{n\ge 0}$ in bases
$\bsp=(p_1,\ldots,p_s)$, and fixed $\bssigma\in [0,1)^s$, we
denote by $\cHal_{\bsp,\bssigma}$ the sequence
$(\bsx_n\oplus_{\bsp}\bssigma)_{n\ge 0}$.

In the following we are going to study the mean square worst-case
error of QMC integration in $\cH (K_{s,\bsgamma})$ by randomly
$\bsp$-adically shifted Halton sequences, where the mean is taken
with respect to a random shift $\bssigma$. I.e., we study
\[\EE_{\bssigma} [e_{N,s}^2 (\cHal_{\bsp,\bssigma},K_{s,\bsgamma})],\]
where $e_{N,s}^2 (\cHal_{\bsp,\bssigma},K_{s,\bsgamma})$ denotes
the squared worst-case error of integration using a QMC rule based
on the first $N$ points of $\cHal_{\bsp,\bssigma}$, and where
$\EE_{\bssigma} $ denotes expectation with respect to a
$\bsp$-adic shift $\bssigma=(\sigma_1,\ldots,\sigma_s)$, where the
$\sigma_j$ are independent and uniformly distributed in $[0,1)$.

The mean square error can then be analyzed using the so-called
$\bsp$-adic shift-invariant kernel $K_{\sh}$ of the space $\cH
(K_{s,\bsgamma})$, which is defined by
\begin{equation}\label{eqdefksh}
K_{\sh}(\bsx,\bsy)=\int_{[0,1]^s} K_{s,\bsgamma}(\bsx\oplus_{\bsp} \bssigma,\bsy\oplus_{\bsp}\bssigma)\rd\bssigma.
\end{equation}

The relevance of the $\bsp$-adic shift invariant kernel is visible
in the following formula, which follows by using standard methods
for the worst-case error of integration in reproducing kernel
Hilbert spaces as in~\cite[Theorem~12.4]{DP05} or in \cite{HickWo}
\begin{equation}\label{eqshifterror}
 \EE_{\bssigma}[e_{N,s}^2 (\cHal_{\bsp,\bssigma},K_{s,\bsgamma})]=
 e_{N,s}^2 (\cHal_{\bsp},K_{\sh}),
\end{equation}
where $e_{N,s}^2 (\cHal_{\bsp},K_{\sh})$ is the worst-case error
of a QMC rule using the unshifted Halton sequence $\cHal_{\bsp}$
in the reproducing kernel Hilbert space with reproducing kernel
$K_{\sh}$. It is well known (see, for example, \cite{SW98} or
\cite[Proposition~2.11]{DP10}) that for every sequence
$\mathcal{S}=(\bsx_n)_{n \ge 0}\subseteq [0,1)^s$ and every reproducing kernel $K$
on $[0,1]^{2s}$ we have
\begin{eqnarray}\label{fr_wc_err}
e_{N,s}^2(\mathcal{S},K) & = & \int_{[0,1]^s}\int_{[0,1]^s} K(\bsx,\bsy) \rd \bsx \rd \bsy
-\frac{2}{N}\sum_{n=0}^{N-1} \int_{[0,1]^s} K(\bsx_n,\bsy)\rd \bsy\nonumber\\
& & + \frac{1}{N^2} \sum_{n,m=0}^{N-1} K(\bsx_m,\bsx_n).
\end{eqnarray}

This means that we have to compute the kernel $K_{\sh}$ defined in~\eqref{eqdefksh}.
The following proposition shows that the
$\bsp$-adic shift invariant kernel can be represented in terms of
$\bsp$-adic functions

\begin{proposition}\label{fo_shinv_kernel_a}
Let $K_{s,\bsgamma}$ be the reproducing kernel given by
\eqref{ker_sob_space}. Then the corresponding $\bsp$-adic shift
invariant kernel is given by
\begin{eqnarray}\label{ds_kernel}
K_{\sh}(\bsx,\bsy) =  \sum_{\bsk \in \NN_0^s} r_{\bsp
,\bsgamma}(\bsk)  \beta_{\bsk}(\bsx)
\overline{\beta_{\bsk}(\bsy)},
\end{eqnarray}
where for $\bsk=(k_1,\ldots,k_s)$ we put $r_{\bsp ,\bsgamma}(\bsk)
= \prod_{j=1}^s r_{p_j,\gamma_j}(k_j)$, and for $k = \kappa_{a-1}
p^{a-1} + \cdots + \kappa_1 p + \kappa_0$ with $\kappa_{a-1} \neq
0$, we put
\[
  r_{p,\gamma}(k)
= \left\{ \begin{array}{ll}
        1+\frac{\gamma}{3} & \mbox{ if } k = 0, \\
        \frac{\gamma}{2p^{2a}}\left(\frac{1}{\sin^2(\kappa_{a-1}\pi/p)} - \frac{1}{3} \right) & \mbox{ if } k > 0.
        \end{array}\right.
\]
\end{proposition}

The proof of Proposition~\ref{fo_shinv_kernel_a} is deferred to
the Appendix.

\subsection{The mean square worst-case error of $\bsp$-adically shifted Halton sequences}\label{secuppersob}

Let
\[\widehat{e}_{N,s}(\cHal_{\bsp},K_{s,\bsgamma}):=\sqrt{\EE_{\bssigma} [e^2_{N,s}(\cHal_{\bsp,\bssigma},K_{s,\bsgamma})]}\]
denote the root mean square (RMS) worst-case error of QMC integration using the first
$N$ points of the randomly $\bsp$-adically shifted Halton sequence
$\cHal_{\bsp}$.

Combining Equations~\eqref{eqshifterror} and \eqref{fr_wc_err} and
Proposition~\ref{fo_shinv_kernel_a}, and by some straightforward calculations, we obtain the following
result.
\begin{proposition}\label{propehat}
We have
\begin{equation*}\label{eqehat}
 [\widehat{e}_{N,s}(\cHal_{\bsp},K_{s,\bsgamma})]^2= \sum_{\bsk\in\NN_0^s\setminus\{\bszero\}}r_{\bsp,\bsgamma}(\bsk) \left|\frac{1}{N}\sum_{n=0}^{N-1}
 \beta_{\bsk}(\bsx_n)\right|^2,
\end{equation*}
where the $\bsx_n$ are the points of the unshifted Halton sequence
$\cHal_{\bsp}$, and where $r_{\bsp,\bsgamma}$ is defined as in
Proposition~\ref{fo_shinv_kernel_a}.
\end{proposition}

We are now going to use Proposition~\ref{propehat} to obtain error
bounds for the RMS worst-case error; we have the following result.
\begin{theorem}\label{thmerrsob}
We have
\begin{eqnarray*}
[\widehat{e}_{N,s}(\cHal_{\bsp},K_{s,\bsgamma})]^2 \le
\frac{1}{N^2}\left[\prod_{j=1}^s \left(1+\gamma_j (\log N)
\frac{p_j^2}{\log p_j}\right)+
\prod_{j=1}^s\left(1+\frac{\gamma_j}{2}\right) \prod_{j=1}^s
\left(1+\frac{\gamma_j p_j}{6}\right)\right].
\end{eqnarray*}
In particular, if $\sum_{j = 1}^\infty \gamma_j \frac{p_j^2}{\log
p_j} < \infty$, then for any $\delta >0$ we have
$$\widehat{e}_{N,s}(\cHal_{\bsp},K_{s,\bsgamma}) \ll_{\delta,\bsgamma,\bsp} \frac{1}{N^{1-\delta}}$$
where the implied constant is independent of the dimension $s$.
\end{theorem}

If $p_j$ is the $j$-th prime number, then we know from the Prime
Number Theorem that $p_j \sim j \log j$ for $j \rightarrow
\infty$ and hence the condition $\sum_{j \ge 1} \gamma_j
\frac{p_j^2}{\log p_j} < \infty$ is equivalent to $\sum_{j \ge 1}
\gamma_j \frac{(j \log j)^2}{\log (j \log j)} < \infty$. Hence we
obtain the following corollary to Theorem~\ref{thmerrsob}:
\begin{corollary}
If $$\sum_{j \ge 1} \gamma_j j^2 \log j < \infty,$$ then for any
$\delta >0$ we have
$$\widehat{e}_{N,s}(\cHal_{\bsp},K_{s,\bsgamma}) \ll_{\delta,\bsgamma,\bsp} \frac{1}{N^{1-\delta}}$$
where the implied constant is independent of the dimension $s$.
Consequently, the corresponding QMC rule achieves a ``strong
polynomial tractability'' error bound and the
$\varepsilon$-exponent is equal to one, which is optimal in view of the fact that we are using an open type QMC rule (see our remark in the introduction).
\end{corollary}

Note that the condition on the weights here is slightly better than the condition \eqref{weightcondwang}
of Wang for unshifted Halton sequences. For example, it is satisfied for weights of the form
$\gamma_j=j^{-(3+\epsilon)}$ with arbitrarily small but positive $\epsilon>0$. Such weights do not satisfy Wang's condition \eqref{weightcondwang}.\\

For the proof of Theorem~\ref{thmerrsob} we need the following
technical lemmas. The first one is \cite[Lemma~1]{P13}, which uses the fact that the $p_1,\ldots,p_s$ are
pairwise different primes.

\begin{lemma}[{\cite[Lemma~1]{P13}}]\label{bd_gammasum}
Let, for $n\in\NN_0$,
$\bsx_n=(\phi_{p_1}(n),\ldots,\phi_{p_s}(n))$ be the $n$-th point
of the Halton sequence $\cHal_{\bsp}$. Then for any $N \in \NN$
and any $\bsk \in \NN_0^s\setminus\{\bszero\}$ we have
$$\left|\sum_{n=0}^{N-1}\beta_{\bsk}(\bsx_n)\right| \le \frac{1}{|\sin(\pi \sum_{j=1}^s \phi_{p_j}(k_j))|}.$$
\end{lemma}

We introduce the following notation: for $\bsg=(g_1,\ldots,g_s)\in
\NN^s$ let
$$\Delta_{\bsp}(\bsg)=\{\bsk=(k_1,\ldots,k_s)\in \NN_0^s\, : \, 0 \le k_i < p_i^{g_i}\mbox{ for all } 1 \le i \le s\}$$
and let
$$\overline{\Delta}_{\bsp}(\bsg)=\{\bsk=(k_1,\ldots,k_s)\in
\NN^s\, : \, 1 \le k_i < p_i^{g_i}\mbox{ for all } 1 \le i \le
s\}.$$ Furthermore, let
$\Delta^{\ast}(\bsg)=\Delta(\bsg)\setminus\{\bszero\}$.

We also need the following lemma.
\begin{lemma}\label{le2}
We have $$\sum_{\bsk \in \NN_0^s}
r_{\bsp,\bsgamma}(\bsk)=\prod_{j=1}^s
\left(1+\frac{\gamma_j}{2}\right) \ \ \ \mbox{ and } \ \ \
\sum_{\bsk \in \Delta_{\bsp}(\bsg)}
r_{\bsp,\bsgamma}(\bsk)=\prod_{j=1}^s\left(1+\frac{\gamma_j}{3}
+\frac{\gamma_j}{6}\left(1-\frac{1}{p_j^{g_j}}\right)\right).$$
\end{lemma}

\begin{proof}
This is easy calculation. Since $r_{\bsp,\bsgamma}$ is of product
form it suffices to show the one-dimensional case (to keep notation simple, we omit the
indices $j$ denoting the components).

To show the first formula, we split the summation over the $k \in
\NN$ into summations over the domains $\{p^u,p^u+1,\ldots,p^{u+1}-1\}$ for $u \in \NN_0$. More precisely,
\begin{eqnarray*}
\sum_{k=0}^{\infty} r_{p,\gamma}(k) = 1+ \frac{\gamma}{3}
+\sum_{u=0}^{\infty} \frac{\gamma}{2 p^{2(u+1)}}
\sum_{k=p^u}^{p^{u+1}-1} \left(\frac{1}{\sin^2(\kappa_u
\pi/p)}-\frac{1}{3}\right),
\end{eqnarray*}
where $\kappa_u$ is the most significant digit in the $p$-adic
expansion of $k \in \{p^u,p^u+1,\ldots,p^{u+1}-1\}$. Then
\begin{eqnarray*}
\sum_{k=0}^{\infty} r_{p,\gamma}(k) = 1+ \frac{\gamma}{3}
+\sum_{u=0}^{\infty} \frac{\gamma}{2 p^{2(u+1)}} \left( p^u
\sum_{\kappa_u=1}^{p-1} \frac{1}{\sin^2(\kappa_u \pi/p)}
-\frac{p^u(p-1)}{3}\right).
\end{eqnarray*}
Now we use the formula $$\sum_{\kappa=1}^{p-1}
\frac{1}{\sin^2(\kappa \pi/p)}=\frac{p^2-1}{3}$$ (see, for
example, \cite[Corollary~A.23]{DP10}). Then we have
\begin{eqnarray*}
\sum_{k=0}^{\infty} r_{p,\gamma}(k) & = & 1+ \frac{\gamma}{3} +\sum_{u=0}^{\infty} \frac{\gamma}{2 p^{u+2}} \left(\frac{p^2-1}{3}   -\frac{p-1}{3}\right)\\
& = & 1+\frac{\gamma}{3} +\frac{\gamma (p-1)}{6 p} \sum_{u=0}^{\infty} \frac{1}{p^u}\\
& = & 1+\frac{\gamma}{2}.
\end{eqnarray*}
The second formula can be shown in the same way with the only
difference that the sum over $u$ is now restricted to $u \in
\{0,1,\ldots,g-1\}$.
\end{proof}

Now we give the proof of Theorem~\ref{thmerrsob}:

\begin{proof}
Using Proposition~\ref{propehat} and the fact that
$|\beta_{\bsk}(\bsx)| \le 1$ we have, for $\bsg\in\NN^s$,
\begin{equation*}
[\widehat{e}_{N,s}(\cHal_{\bsp},K_{s,\bsgamma})]^2\le
\sum_{\bsk \in \Delta_{\bsp}^{\ast}(\bsg)} r_{\bsp,\bsgamma}(\bsk)
\left|\frac{1}{N}\sum_{n=0}^{N-1}\beta_{\bsk}(\bsx_n)\right|^2 +
\sum_{\bsk \in \NN_0^s \setminus \Delta_{\bsp}(\bsg)}
r_{\bsp,\bsgamma}(\bsk).
\end{equation*}
Using Lemma~\ref{le2} we have
\begin{eqnarray*}
\sum_{\bsk \in \NN_0^s \setminus \Delta_{\bsp}(\bsg)}
r_{\bsp,\bsgamma}(\bsk) & =
& \sum_{\bsk \in \NN_0^s} r_{\bsp,\bsgamma}(\bsk) - \sum_{\bsk \in \Delta_{\bsp}(\bsg)} r_{\bsp,\bsgamma}(\bsk)\\
& = & \prod_{j=1}^s\left(1+\frac{\gamma_j}{2}\right)-\prod_{j=1}^s\left(1+\frac{\gamma_j}{3} +\frac{\gamma_j}{6}\left(1-\frac{1}{p_j^{g_j}}\right)\right)\\
& = & \prod_{j=1}^s\left(1+\frac{\gamma_j}{2}\right) \left[1-\prod_{j=1}^s\left(1- \frac{\gamma_j}{6 p_j^{g_j}\left(1+\frac{\gamma_j}{2}\right)}\right)\right]\\
& \le &
\prod_{j=1}^s\left(1+\frac{\gamma_j}{2}\right)\left[1-\prod_{j=1}^s\left(1-
\frac{\gamma_j}{6 p_j^{g_j}}\right)\right].
\end{eqnarray*}

Let $$\Sigma:=\sum_{\bsk \in \Delta_{\bsp}^{\ast}(\bsg)}
r_{\bsp,\bsgamma}(\bsk)
\left|\frac{1}{N}\sum_{n=0}^{N-1}\beta_{\bsk}(\bsx_n)\right|^2.$$
We can use Lemma~\ref{bd_gammasum} to obtain
\begin{equation}\label{Sigmaest0}
\Sigma \le  \frac{1}{N^2} \sum_{\emptyset \not= \uu \subseteq [s]}
\sum_{\bsk_{\uu} \in \overline{\Delta}_{\bsp_{\uu}}(\bsg_{\uu})}
\frac{\prod_{j \in \uu}r_{p_j,\gamma_j}(k_j)}{|\sin(\pi \sum_{j
\in \uu}\phi_{p_j}(k_j))|^2},
\end{equation}
where $\bsp_{\uu}$ is the vector of those $p_j$ for which $j\in\uu$ (and similarly for $\bsg_{\uu}$).

Next we show that
\begin{equation}\label{Sigmaest1}
\sum_{\bsk_{\uu} \in \overline{\Delta}_{\bsp_{\uu}}(\bsg_{\uu})}
\frac{\prod_{j \in \uu}r_{p_j,\gamma_j}(k_j)}{|\sin(\pi \sum_{j
\in \uu}\phi_{p_j}(k_j))|^2} \le \frac{1}{3} \prod_{j\in \uu}
\frac{\gamma_j g_j p_j^2}{2}.
\end{equation}
In order to keep the notation simple we assume that $\uu=[t]$ (the
argument works in the same way for general $\uu \subseteq [s]$).
We have
\begin{eqnarray*}
\lefteqn{\sum_{\bsk_{[t]} \in
\overline{\Delta}_{\bsp_{[t]}}(\bsg_{[t]})}
\frac{\prod_{j=1}^t r_{p_j,\gamma_j}(k_j)}{|\sin(\pi \sum_{j =1}^t \phi_{p_j}(k_j))|^2}}\\
& = & \sum_{u_1=0}^{g_1 -1}\ldots \sum_{u_t=0}^{g_t -1}
\sum_{k_1=p_1^{u_1}}^{p_1^{u_1 +1}-1}\ldots
\sum_{k_t=p_t^{u_t}}^{p_t^{u_t +1}-1} \prod_{j=1}^t
\frac{\gamma_j}{2 p_j^{2 (u_j+1)}} \left(\frac{1}{\sin^2(\pi
\kappa_{j,u_j}/p_j)}-\frac{1}{3}\right)\frac{1}{|\sin(\pi  \sum_{j
=1}^t \phi_{p_j}(k_j))|^2},
\end{eqnarray*}
where $\kappa_{j,u_j}$ is the most significant digit of $k_j\in\{p_j^{u_j},\ldots,p_j^{u_j+1}-1\}$.

It has been shown in \cite[p.~181]{DP05} that
$$\frac{1}{\sin^2(\pi \kappa_{j,u_j}/p_j)}-\frac{1}{3} \le \frac{4
p_j^2-9}{27} \le p_j^2.$$ Hence we have
\begin{eqnarray*}
\lefteqn{\sum_{\bsk_{[t]} \in
\overline{\Delta}_{\bsp_{[t]}}(\bsg_{[t]})}
\frac{\prod_{j=1}^t r_{p_j,\gamma_j}(k_j)}{|\sin(\pi \sum_{j =1}^t \phi_{p_j}(k_j))|^2}}\\
& \le & \sum_{u_1=0}^{g_1 -1}\ldots \sum_{u_t=0}^{g_t -1}
\prod_{j=1}^t \frac{\gamma_j}{2 p_j^{2 u_j}}
\sum_{k_1=p_1^{u_1}}^{p_1^{u_1 +1}-1}\ldots
\sum_{k_t=p_t^{u_t}}^{p_t^{u_t +1}-1}  \frac{1}{|\sin(\pi  \sum_{j
=1}^t \phi_{p_j}(k_j))|^2}.
\end{eqnarray*}
Now we use the estimate
\begin{equation}\label{bdt}
\sum_{k_1=p_1^{u_1}}^{p_1^{u_1 +1}-1}\ldots
\sum_{k_t=p_t^{u_t}}^{p_t^{u_t +1}-1}\frac{1}{|\sin(\pi
\sum_{j=1}^t \phi_{p_j}(k_j))|^2} \le \frac{1}{3} \prod_{j =1}^t
p_j^{2 u_j+2},
\end{equation}
for nonnegative integers $u_1,\ldots,u_t$, the proof of which follows
exactly the lines of \cite[Proof of Eq. (8)]{P13} and obtain
\begin{eqnarray*}
\sum_{\bsk_{[t]} \in \overline{\Delta}_{\bsp_{[t]}}(\bsg_{[t]})}
\frac{\prod_{j=1}^t r_{p_j,\gamma_j}(k_j)}{|\sin(\pi \sum_{j =1}^t
\phi_{p_j}(k_j))|^2}\le  \frac{1}{3} \prod_{j=1}^t \frac{\gamma_j
g_j p_j^2}{2}.
\end{eqnarray*}
Hence \eqref{Sigmaest1} is shown.

Now, inserting \eqref{Sigmaest1} into \eqref{Sigmaest0} gives
\begin{eqnarray*}
\Sigma \le \frac{1}{N^2} \prod_{j=1}^s \left(1+\frac{\gamma_j g_j
p_j^2}{2}\right).
\end{eqnarray*}
Therefore,
\begin{equation*}
 [\widehat{e}_{N,s}(\cHal_{\bsp},K_{s,\bsgamma})]^2
\le \frac{1}{N^2} \prod_{j=1}^s \left(1+\frac{\gamma_j g_j
p_j^2}{2}\right) +
\prod_{j=1}^s\left(1+\frac{\gamma_j}{2}\right)\left[1-\prod_{j=1}^s\left(1-
\frac{\gamma_j}{6 p_j^{g_j}}\right)\right].
\end{equation*}
By choosing $g_j=\lfloor 2 \log_{p_j} N \rfloor$ we obtain
\begin{eqnarray*}
\lefteqn{[\widehat{e}_{N,s}(\cHal_{\bsp},K_{s,\bsgamma})]^2}\\
&\le&  \frac{1}{N^2} \prod_{j=1}^s \left(1+\gamma_j (\log N) \frac{p_j^2}{\log p_j}\right)+
\prod_{j=1}^s\left(1+\frac{\gamma_j}{2}\right)\left[1-\prod_{j=1}^s\left(1- \frac{\gamma_j p_j}{6 N^2}\right)\right]\\
& \le & \frac{1}{N^2}\left[\prod_{j=1}^s \left(1+\gamma_j (\log N)
\frac{p_j^2}{\log p_j}\right) +
\prod_{j=1}^s\left(1+\frac{\gamma_j}{2}\right) \prod_{j=1}^s
\left(1+\frac{\gamma_j p_j}{6}\right)\right]
\end{eqnarray*}
as desired.

We have
\begin{eqnarray*}
\prod_{j=1}^s\left(1+\frac{\gamma_j}{2}\right) \prod_{j=1}^s
\left(1+\frac{\gamma_j p_j}{6}\right) & = {\rm e}^{\sum_{j=1}^s
\left(\log \left(1+\frac{\gamma_j}{2}\right) +
\log\left(1+\frac{\gamma_j p_j}{6}\right)\right)} \le {\rm
e}^{\sum_{j=1}^s \gamma_j p_j}.
\end{eqnarray*}
Hence, if $\sum_{j=1}^{\infty} \gamma_j p_j < \infty$ then we
obtain
$$\prod_{j=1}^{\infty}\left(1+\frac{\gamma_j}{2}\right) \prod_{j=1}^{\infty} \left(1+\frac{\gamma_j p_j}{6}\right)< \infty.$$
To estimate the term
\[\prod_{j=1}^s \left(1+\gamma_j \frac{p_j^2}{\log p_j} (\log N)\right)\]
we use an argument by Hickernell and Niederreiter
(see~\cite[Lemma~3]{HN03} or \cite[p.~222]{DP10}), which implies
that for any $\delta>0$ there is a positive constant
$C_{\delta,\bsgamma,\bsp}>0$ with the property that
\[\prod_{j=1}^\infty \left(1+\gamma_j \frac{p_j^2}{\log p_j} (\log N)\right)\le C_{\delta,\bsgamma,\bsp}N^\delta\]
provided that $\sum_{j=1}^\infty \gamma_j \frac{p_j^2}{\log
p_j}<\infty$. Combining these observations, the second assertion
in the theorem follows as well.
\end{proof}

\section{The RMS $L_2$-discrepancy of the $\bsp$-adically shifted Halton sequence}\label{secRMSL2}

As already mentioned, the results from Theorem~\ref{thmerrsob} on
the RMS worst-case error of integration in $\cH (K_{s,\bsgamma})$
using Halton sequences carry over to the RMS weighted
$L_2$-discrepancy. Let us consider the unweighted case
$\bsgamma=\bsone$, i.e., $\gamma_j=1$ for all $j \ge 1$.

It follows from a result of Roth~\cite{Roth} that for any
$N$-element point set $\cP$ in $[0,1)^s$ the $L_2$-discrepancy is
at least of order $$L_{2,N,\bsone}(\cP) \gg_s \frac{(\log
N)^{\frac{s-1}{2}}}{N}$$ and this is optimal for general point
sets (see \cite{Roth2,Roth4}). Explicit constructions of point
sets whose $L_2$-discrepancies achieve this lower bound were given
by Chen and Skriganov~\cite{CS02} and by Dick and
Pillichshammer~\cite{DP14}. Roth's result for finite point sets
was extended to infinite sequences by Proinov~\cite{pro86} who
showed that for every sequence $\mathcal{S}$ in $[0,1)^s$ we have
\begin{equation}\label{pro}
L_{2,N,\bsone}(\mathcal{S}) \gg_s \frac{(\log N)^{\frac{s}{2}}}{N}
\ \ \ \mbox{ for infinitely many $N \in \NN$,}
\end{equation}
and this is again optimal. An explicit construction of infinite
sequences whose $L_2$-discrepancies achieve this lower bound was
given by Dick and Pillichshammer~\cite{DP14}.

On the other hand, for the Halton sequence $\cHal_{\bsp}$  in
dimension $s$ with mutually relative prime bases
$\bsp=(p_1,\ldots,p_s)$ it is known that
$$L_{2,N,\bsone}(\cHal_{\bsp}) \ll_s \frac{(\log N)^s}{N}.$$ This
can be deduced from the estimate for the star-discrepancy of the
Halton sequence, see, for example, \cite{DP10,LP14,N92}.

Let $\widehat{L}_{2,N,\bsone}(\cHal_{\bsp})$ denote the RMS
$L_2$-discrepancy of the $\bsp$-adically shifted Halton sequence
$\cHal_{\bsp}$. Then Theorem~\ref{thmerrsob} implies the following
result.

\begin{corollary}
The RMS $L_2$-discrepancy of the $\bsp$-adically shifted Halton
sequence satisfies
\begin{equation}\label{bdavL2}
\widehat{L}_{2,N,\bsone}(\cHal_{\bsp}) \ll_s \frac{(\log
N)^{\frac{s}{2}}}{N}.
\end{equation}
According to \eqref{pro} this bound is best possible in the order
of magnitude in $N$.

In particular, for every $N \ge 2$ there exists a $\bsp$-adic
shift $\bssigma_*$ (which depends on $N$), such that the
$L_2$-discrepancy of $\cHal_{\bsp,\bssigma_*}$ satisfies
$$L_{2,N,\bsone}(\cHal_{\bsp,\bssigma_*}) \ll_s \frac{(\log
N)^{\frac{s}{2}}}{N}.$$
\end{corollary}

\section{Concluding Remarks}\label{final}

As already pointed out, we could have chosen any anchor $\bsw \in
[0,1]^s$ instead of the anchor $\bsone$ in the definition of the
Sobolev space (see, for example, \cite{DKPS05}). All the formulas
and obtained results would then be more or less the same as the
ones obtained here, with the only difference that (in dimension
one) $r_{p,\gamma,w}(0) =1+\gamma(w^2-w+1/3)$ if $w$ is the
anchor. For $k >0$ the value of $r_{p,\gamma,w}(k)$ is invariant
with respect to changing $w$ and coincides with the one for $r_{p,\gamma}(k)$ (i.e.
$w=0$) as given in Proposition~\ref{fo_shinv_kernel_a}.

Likewise we could have also analyzed the unanchored weighted
Sobolev space whose kernel is given by
$$K(\bsx,\bsy)=\prod_{j=1}^s
\left(1+\gamma_j\left(\frac{B_2(\{x_j-y_j\})}{2}+\left(x_j-\frac{1}{2}\right)\left(y_j-\frac{1}{2}\right)\right)\right),$$
where $B_2(x)=x^2-x+\tfrac{1}{6}$ is the second Bernoulli
polynomial. The inner product in this space is given by $$\langle f,g\rangle = \sum_{\uu
\subseteq [s]} \gamma_{\uu}^{-1} \int_{[0,1]^{|\uu|}}
\left(\int_{[0,1]^{s-|\uu|}} \frac{\partial^{|\uu|} f}{\partial
\bsx_{\uu}}(\bsx) \rd \bsx_{[s]\setminus \uu}\right)
\left(\int_{[0,1]^{s-|\uu|}} \frac{\partial^{|\uu|} g}{\partial
\bsx_{\uu}}(\bsx) \rd \bsx_{[s]\setminus \uu}\right) \rd
\bsx_{\uu},$$
where $\bsx_{[s]\setminus \uu}$ denotes the projection of $\bsx$ onto those components with $j\notin\uu$.

In this case the coefficients $r_{p,\gamma}$ in the series
expansion of the corresponding $\bsp$-adic shift invariant kernel
are given as
\[
  r_{p,\gamma}(k)
= \left\{ \begin{array}{ll}
        1 & \mbox{ if } k = 0, \\
        \frac{\gamma}{2p^{2a}}\left(\frac{1}{\sin^2(\kappa_{a-1}\pi/p)} - \frac{1}{3} \right)
& \mbox{ if } k=\kappa_{a-1}p^{a-1}+\cdots+\kappa_1 p+\kappa_0\\
&  \mbox{ with $\kappa_j \in \{0,1,\ldots,p-1\}$ and
$\kappa_{a-1}\not=0$},
        \end{array}\right.
\]
(cf. Proposition~\ref{fo_shinv_kernel_a}).\\

Finally, let us point out that the Sobolev space
$\cH(K_{s,\bsgamma})$ under consideration is closely linked to
another weighted reproducing kernel Hilbert space $\cH
(K_{s,\bsp,\bsalpha,\bsgamma})$ based on the $\bsp$-adic function
system $B_{\bsp}^s$. The latter is analogous to a so-called
Korobov space, which is based on the trigonometric function system
(see, e.g.,~\cite{K03,SW01}), and to a Walsh space, which is based
on the Walsh function system (see, e.g.,~\cite{DKPS05,DP05}). The
kernel $K_{s,\bsp,\bsalpha,\bsgamma}$ is of product form, i.e.,
$K_{s,\bsp,\bsalpha,\bsgamma}(\bsx,\bsy)=\prod_{j=1}^s
K_{p_j,\alpha_j,\gamma_j}(x_j,y_j)$, where $\bsp,\bsgamma$ are as
before and where $\bsalpha=(\alpha_1,\ldots,\alpha_s)$ with
$\alpha_j >1$. In dimension one the kernel is given by
$$K_{p,\alpha,\gamma}(x,y) =  \sum_{k=0}^\infty r_{p,\alpha,\gamma}(k) \beta_k(x) \overline{\beta_k(y)}\ \ \ \mbox{ for all $ x,y\in [0,1)$,}
$$
where
\begin{equation*}
r_{p,\alpha,\gamma}(k) :=  \left\{
\begin{array}{ll}
1 & \mbox{if } k = 0,\\
\gamma p^{-\alpha \lfloor \log_p (k) \rfloor} & \mbox{if } k \neq
0.
\end{array}
\right.
\end{equation*}
The special case $\cH (K_{p,2,1})$ (i.e., $\alpha=2$ and
$\gamma=1$) has been introduced and analyzed in \cite{P13}. It can
be shown that $K_{p,\alpha,\gamma}=K_{\wal,\gamma}$, where
$K_{\wal,\gamma}$ is the reproducing kernel of the Walsh space of
univariate functions defined in \cite[Section~2.2]{DP05} (see also
\cite{DKPS05}). The corresponding inner product of two functions
$f$ and $g$ on $[0,1]$ is defined by
\[
   \langle f, g \rangle_{p,\alpha,\gamma}
:= \sum_{k \in \NN_0} r_{p,\alpha,\gamma}(k)^{-1} \widehat{f}(k)
\; \overline{\widehat{g}(k)},\ \ \mbox{ where }\ \
   \widehat{f}(k) = \int_0^1 f(x) \; \overline{\beta_{k}(x)} \rd x.
\]

The space $\cH (K_{s,\bsp,\bsalpha,\bsgamma})$ contains functions
that can be represented by $\bsp$-adic function series. For this
particular space we can derive results for unshifted Halton
sequences, i.e., we do not need any randomization method for our
analysis. As in the proof of Theorem~\ref{thmerrsob} one can show
the following result:
\begin{theorem}\label{therrexpl2}
Let $\bsalpha=(\alpha_1,\ldots,\alpha_s)$ with $\alpha_j\ge 2$.
Then for any $N \in \NN$ we have
\begin{equation}\label{korbd1}
e^2_{N,s}(\cHal_{\bsp},K_{s,\bsp,\bsalpha,\bsgamma}) \le
\frac{1}{N^2} \left[\prod_{j=1}^s \left(1+2 \gamma_j p_j^2 (\log
N)\right) +\prod_{j=1}^s\left(1+\gamma_j p_j\right)\prod_{i=1}^s
(1+\gamma_j p_j^2)\right].
\end{equation}
In particular, if $\sum_{j=1}^\infty \gamma_j p_j^2 < \infty$,
then for any $\delta>0$ we have
$$e_{N,s}(\cHal_{\bsp},K_{s,\bsp,\bsalpha,\bsgamma}) \ll_{\delta,\bsgamma,\bsp}  \frac{1}{N^{1-\delta}},$$
where the implied constant is independent of the dimension $s$.

If all $\alpha_j >2$, then \eqref{korbd1} can be improved to
$$e^2_{N,s}(\cHal_{\bsp},K_{s,\bsp,\bsalpha,\bsgamma}) \le \frac{1}{N^2} \left(-1+\prod_{j=1}^s \left(1+2\gamma_j p_j^2\right)\right).$$
\end{theorem}

For the kernel $K_{s,\bsp,\bsalpha,\bsgamma}$ it can be easily
shown that condition \eqref{condHickKri} is satisfied with
$c_{K_{s,\bsp,\bsalpha,\bsgamma}}=-1+\prod_{j=1}^s \big(1+\gamma_j
\tfrac{p_j^{\alpha_j}(p_j-1)}{p_j^{\alpha_j}-p_j}\big)$. Hence the
obtained convergence rates are optimal in the order of magnitude
in $N$. Apart from providing an optimal convergence rate, we would like to
stress that the integration rule used in Theorem~\ref{therrexpl2} is fully explicit:
no construction algorithm or randomization is needed.

\section*{Appendix: Computation of the $\bsp$-adic shift invariant kernel}

In this Appendix we are going to prove
Proposition~\ref{fo_shinv_kernel_a}. To begin, we need some
preparations.

A first auxiliary result is an analogue of
\cite[Corollary~A.13]{DP10} which is easily shown in the same
fashion, using that the $\bsp$-adic shift is measure preserving.
\begin{lemma}\label{lemintshift}
Let $\bssigma\in [0,1)^s$ and let $\bsp$ be a vector of integers
greater than or equal to 2. Then for all $f\in L_2([0,1]^s)$ we
have
\[\int_{[0,1]^s} f(\bsx)\rd\bsx=\int_{[0,1]^s} f(\bsx\oplus_{\bsp}\bssigma)\rd\bsx.\]
\end{lemma}

Furthermore, we need the following result which is the $p$-adic
and slightly generalized version of \cite[Theorem A.2]{DP10}. To this end, we define for
$f:[0,1]^2\To\CC$ and $k,l\in\NN_0^s$,
$$\widehat{f}(k,l):=\int_0^1 \int_0^1 f(x,y) \overline{\beta_{k}(x)}\beta_l (y)\rd x\rd y.$$

\begin{lemma}\label{lemabsconv}
Let $f:[0,1]^2\To\CC$, let $p$ be a prime number, and assume that
the following assumptions are satisfied:
\begin{itemize}
 \item $\sum_{k,l=0}^\infty |\widehat{f}(k,l)|<\infty$,
 \item $f(x,y)$ is continuous in $(x,y)$ if $x$ and $y$ are $p$-adic irrationals,
 \item $f(x,y)$ is continuous from above in $(x,y)$ if $x$ and $y$ are $p$-adic rationals.
 \item $f(x,y)$ is continuous in the first variable and continuous from above in the second variable
 if $x$ is a $p$-adic irrational and $y$ is a $p$-adic rational.
 \item $f(x,y)$ is continuous from above in the first variable and continuous in the second variable
 if $x$ is a $p$-adic rational and $y$ is a $p$-adic irrational.
\end{itemize}

Then $\sum_{k,l=0}^\infty \widehat{f}(k,l) \beta_k (x)
\overline{\beta_l(y)}$ converges uniformly to $f(x,y)$, and we
have
\[f(x,y)=\sum_{k,l=0}^\infty \widehat{f}(k,l)\beta_k (x)\overline{\beta_l(y)}\ \ \ \mbox{ for all $x,y\in [0,1).$}\]
\end{lemma}
\begin{proof}
We proceed similarly to the proof of \cite[Theorem~A.20]{DP10}.
Suppose that all assumptions in the lemma are satisfied, and let
$x,y\in [0,1)$ be given. In the following we write $\mathrm{e}(t)$ for $\mathrm{e}^{2\pi\icomp t}$.

For $u,v \in \NN_0$ consider the
expression
\begin{eqnarray*}
 \sum_{k=0}^{p^u-1}\sum_{l=0}^{p^v-1}\widehat{f}(k,l)\beta_k (x)\overline{\beta_l(y)}&=&
\int_0^1\int_0^1 f(t,s) \sum_{k=0}^{p^u-1}
\beta_k (x)\overline{\beta_k (t)} \sum_{l=0}^{p^v-1} \overline{\beta_l(y)}\beta_l (s)\rd t \rd s\\
 &=&\int_0^1 \int_0^1 f(t,s) \sum_{k=0}^{p^u-1} \mathrm{e} (\phi_p (k) (\phi_p^+(x)-\phi_p^+(t)))\\
&&\hspace{2.2cm}\times\sum_{l=0}^{p^v-1} \mathrm{e} (\phi_p (l)
(\phi_p^+(s)-\phi_p^+(y)))\rd t \rd s.
\end{eqnarray*}
Let $A_u:=[p^{-u} \lfloor p^u x\rfloor, p^{-u}\lfloor p^u x
\rfloor +p^{-u})$, and let $A_v:=[p^{-v} \lfloor p^v y\rfloor,
p^{-v}\lfloor p^v y \rfloor +p^{-v})$. It is then easily checked by
inserting into the definitions of $\phi$ and $\phi^+$ that
\[\sum_{k=0}^{p^u-1} \mathrm{e} (\phi_p (k) (\phi_p^+(x)-\phi_p^+(t)))=\begin{cases}
                                                                        0 & \mbox{if $t\not\in A_u$,}\\
                                    p^u&\mbox{if $t\in A_u$,}
                                                                       \end{cases}
\]
\[\sum_{l=0}^{p^v-1} \mathrm{e} (\phi_p (l) (\phi_p^+(s)-\phi_p^+(y)))=\begin{cases}
                                                                        0 & \mbox{if $s\not\in A_v$,}\\
                                    p^v&\mbox{if $s\in A_v$,}
                                                                       \end{cases}
\]
Consequently,
\[ \sum_{k=0}^{p^u-1}\sum_{l=0}^{p^v-1}\widehat{f}(k,l)\beta_k (x)\overline{\beta_l(y)}=p^{u+v} \int_{A_u}\int_{A_v} f(t,s)\rd t \rd s.\]
Note that if $x$ (and likewise $y$) is a $p$-adic rational, then
$p^{-u} \lfloor p^u x\rfloor=x$ (and likewise $p^{-v} \lfloor p^v
y\rfloor=y$) for sufficiently large $u$ (and likewise $v$). Hence
we can use continuity of $f$ in the $p$-adic irrationals and
continuity from above in the $p$-adic rationals to see that
$\sum_{k=0}^{p^u-1}\sum_{l=0}^{p^v-1}\widehat{f}(k,l)\beta_k
(x)\overline{\beta_l(y)}$ converges to $f(x,y)$ as $u$ and $v$
tend to infinity.

Since we assumed $\sum_{k,l=0}^\infty |\widehat{f}(k,l)|<\infty$,
the partial sums
$\sum_{k=0}^{U}\sum_{l=0}^{V}\widehat{f}(k,l)\beta_k
(x)\overline{\beta_l(y)}$ are a Cauchy sequence and this implies
uniform convergence to $f(x,y)$.
\end{proof}
The next lemma is analogous to \cite[Lemma~12.2]{DP10}.

\begin{lemma}\label{lemshikernel}
Let the reproducing kernel $K_s \in L_2([0,1]^{2s})$ be of product
form $K_s(\bsx,\bsy)=\prod_{j=1}^s K_j(x_j,y_j)$ and continuous.
For $\bsk\in\NN_0^s$ let
\[\widehat{K}_s(\bsk,\bsk)=\int_{[0,1]^s} \int_{[0,1]^s} K_s
(\bsx,\bsy) \overline{\beta_{\bsk}(\bsx)} \beta_{\bsk}(\bsy)
\rd\bsx\rd\bsy.\] If
\[\sum_{\bsk\in\NN_0^s} |\widehat{K}_s(\bsk,\bsk)|<\infty,\]
then the $\bsp$-adic shift-invariant kernel $K_{\sh}$ is given by
\[K_{\sh}(\bsx,\bsy)=\sum_{\bsk\in\NN_0^s} \widehat{K}_s(\bsk,\bsk)\beta_{\bsk}(\bsx)\overline{\beta_{\bsk}(\bsy)}.\]
\end{lemma}

\begin{proof}
 Using the definition of $K_{\sh}$, we see that for $\bsk,\bsk'\in\NN_0^s$
\begin{eqnarray*}
 \widehat{K}_{\sh}(\bsk,\bsk')&:=&\int_{[0,1]^{2s}} K_{\sh}(\bsx,\bsy)\overline{\beta_{\bsk}(\bsx)}\beta_{\bsk'}(\bsy)\rd\bsx\rd\bsy\\
&=&\int_{[0,1]^{2s}} \int_{[0,1]^s} K_s(\bsx\oplus_{\bsp}
\bssigma,\bsy\oplus_{\bsp}\bssigma)\rd\bssigma
\overline{\beta_{\bsk}(\bsx)}\beta_{\bsk'}(\bsy)\rd\bsx\rd\bsy\\
&=&\int_{[0,1]^{s}} \int_{[0,1]^{2s}} K_s(\bsx\oplus_{\bsp}
\bssigma,\bsy\oplus_{\bsp}\bssigma)
\overline{\beta_{\bsk}(\bsx)}\beta_{\bsk'}(\bsy)\rd\bsx\rd\bsy\rd\bssigma.
\end{eqnarray*}
We now define an inverse $\ominus_{\bsp}$ of the digital shift
operator $\oplus_{\bsp}$ as follows: for a point
$\bsx=(x_1,\ldots,x_s)\in [0,1)^s$ and a given
$\bssigma=(\sigma_1,\ldots,\sigma_s)\in [0,1)^s$, we define
$\bsx\ominus_{\bsp}\bssigma\in [0,1)^s$ to be
\[\bsx\ominus_{\bsp}\bssigma=(x_1\ominus_{p_1}\sigma_1,\ldots,x_s\ominus_{p_s}\sigma_s),\]
where $x_j\ominus_{p_j}\sigma_j=\phi_{p_j} (\phi_{p_j}^+ (x_j) -
\phi_{p_j}^+ (\sigma_j))$. Here $\phi_{p_j}^+ (x_j) - \phi_{p_j}^+
(\sigma_j)$ means subtraction in $\ZZ_{p_j}$.

For any prime $p$, it is known that $\phi_p^+ \circ
\phi_p=\mathrm{id}$ on the set of $p$-adic numbers with infinitely many digits different from $p-1$. Hence,
except for a set of measure zero we have
\begin{eqnarray*}
(x\ominus_p \sigma) \oplus \sigma&=& \phi_p (\phi_p^+ (x\ominus_p \sigma) + \phi_p^+ (\sigma))\\
&=&\phi_p (\phi_p^+ (\phi_p (\phi_p^+ (x) - \phi_p^+(\sigma))) + \phi_p^+ (\sigma))\\
&=&\phi_p (\phi_p^+ (x)-\phi_p^+ (\sigma) +\phi_p^+ (\sigma))\\
&=&\phi_p (\phi_p^+ (x))\\
&=&x,
\end{eqnarray*}
and this property can be carried over to the $s$-dimensional case.
Therefore, we can apply Lemma~\ref{lemintshift}, and obtain

\begin{eqnarray*}
 \widehat{K}_{\sh}(\bsk,\bsk')&=&\int_{[0,1]^{s}} \int_{[0,1]^{2s}} K_s(\bsx,\bsy)
\overline{\beta_{\bsk}(\bsx\ominus_{\bsp}\bssigma)}\beta_{\bsk'}(\bsy\ominus_{\bsp}\bssigma)\rd\bsx\rd\bsy\rd\bssigma\\
&=&\int_{[0,1]^{s}} \int_{[0,1]^{2s}} K_s(\bsx,\bsy)
\overline{\beta_{\bsk}(\bsx)}\beta_{\bsk'}(\bsy)\beta_{\bsk}(\bssigma)\overline{\beta_{\bsk'}(\bssigma)}\rd\bsx\rd\bsy\rd\bssigma\\
&=&\int_{[0,1]^{2s}} K_s(\bsx,\bsy)
\overline{\beta_{\bsk}(\bsx)}\beta_{\bsk'}(\bsy)\rd\bsx\rd\bsy\int_{[0,1]^{s}}
\beta_{\bsk}(\bssigma)\overline{\beta_{\bsk'}(\bssigma)}\rd\bssigma\\
&=& \left\{
\begin{array}{ll}
\widehat{K}_s (\bsk,\bsk) & \mbox{ if } \bsk=\bsk',\\
0 & \mbox{ if } \bsk\not=\bsk',
\end{array}\right.
\end{eqnarray*}
where we used Proposition~\ref{ONB_prop} in the last step.

Under the assumption that $\sum_{\bsk\in\NN_0^s}
|\widehat{K}_s(\bsk,\bsk)|<\infty$, we obtain that
\[\sum_{\bsk\in\NN_0^s} \widehat{K}_s(\bsk,\bsk)\beta_{\bsk}(\bsx)\overline{\beta_{\bsk}(\bsy)}\]
converges for all $\bsx,\bsy\in [0,1)^s$.

In the next step we are going to study the continuity properties
of $K_{\sh}$. By definition,
\[K_{\sh}(\bsx,\bsy)=\int_{[0,1]^s} K_s(\bsx\oplus_{\bsp} \bssigma,\bsy\oplus_{\bsp}\bssigma)\rd\bssigma=
\prod_{j=1}^s \int_0^1 K_j(x_j\oplus_{p_j}
\sigma_j,y_j\oplus_{p_j}\sigma_j)\rd\sigma_j.\] Hence, we are
going to restrict ourselves to considering the one-dimensional
case in the following, and in order to keep notation simple, we
are going to omit the indices $j$ for the moment.

We first show that, for fixed $\sigma$, the mapping $x\mapsto
x\oplus_p \sigma$ is continuous from above in any $x\in[0,1)$. We
recall that for $x\in[0,1)$ we always consider the unique base
$p$ representation of $x$ such that the $p$-adic rationals have a
finite expansion. Let $\varepsilon>0$ be given and choose $k$ such
that $p^{-k}<\varepsilon$. We now choose $\delta< p^{-k}$ such that all $y\in
[x,x+\delta)\cap [0,1)$ share the first $k$ base $p$ digits with
$x$. Indeed, such a $\delta$ can always be found, as we show now.

In the first case, suppose that $x$ is a $p$-adic rational. Then the base $p$ representation is finite, i.e., there is
some $R\in\NN$ such that $x=\sum_{r=0}^R \frac{x_r}{p^{r+1}}$. We then choose $\delta=p^{-K-1}$, where $K>\max\{R,k\}$. With this choice of $\delta$, for any
$y\in [x,x+\delta)$ we then have
$$y= \sum_{r=0}^R \frac{x_r}{p^{r+1}} + \sum_{r=K+1}^{\infty} \frac{y_r}{p^{r+1}}$$
for some $y_{K+1},y_{K+2},\ldots$ in $\{0,1,\ldots,p-1\}$.

On the other hand, suppose that $x$ is a $p$-adic irrational. Then $x$ has an infinite base $p$ expansion,
but infinitely many of the digits $x_r$ are different from $p-1$. We then choose an index $K>k$ such that $x_K< p-1$, and again
take $\delta=p^{-K-1}$. If we add a real number less than $\delta$ to $x$, then this has no influence on the first $k$ digits of $x$, since
$x_K< p-1$. Hence, also in this case, for any
$y\in [x,x+\delta)$ $y$ shares the first $k$ digits with $x$. This shows that we can always find a $\delta$ with the properties stated above.

If all $y\in
[x,x+\delta)\cap [0,1)$ share the first $k$ base $p$ digits with
$x$, this however implies that also the integers $\phi_p^+ (x)$
and $\phi_p^+ (y)$ share the first $k$ base $p$ digits, and so
also the first $k$ digits of $\phi_p^+ (x)+\phi_p^+ (\sigma)$ and
$\phi_p^+ (y)+\phi_p^+ (\sigma)$ coincide. Applying $\phi_p$ to
these expressions yields that $\abs{x\oplus_p \sigma - y\oplus_p
\sigma}<\varepsilon$. Hence, continuity from above of $x\mapsto
x\oplus_p \sigma$ is shown.

If $x\in [0,1)$ is a $p$-adic irrational, i.e., has infinitely
many base $p$ digits different from $p-1$, then we can apply the
same procedure as just outlined to show continuity of $x\mapsto
x\oplus_p \sigma$ from below.

Noting that $K$ is a continuous function in both variables, we see
that, for fixed $\sigma$,

\begin{itemize}
\item $K(x\oplus_p \sigma,y\oplus_p \sigma)$ is continuous in
$(x,y)$ if $x$ and $y$ are $p$-adic irrationals,
 \item $K(x\oplus_p \sigma,y\oplus_p \sigma)$ is continuous from above in $(x,y)$ if $x$ and $y$ are $p$-adic rationals.
 \item $K(x\oplus_p \sigma,y\oplus_p \sigma)$ is continuous in the first variable and continuous from above in the second variable
 if $x$ is a $p$-adic irrational and $y$ is a $p$-adic rational.
 \item $K(x\oplus_p \sigma,y\oplus_p \sigma)$ is continuous from above in the first variable and continuous in the second variable
 if $x$ is a $p$-adic rational and $y$ is a $p$-adic irrational.
\end{itemize}

Note furthermore, that $K(x\oplus_p \sigma,y\oplus_p \sigma)$ is
bounded, so we can apply the dominated convergence theorem to
obtain that $K_{\sh}^{(1)}:=\int_0^1 K(x\oplus_p \sigma,y\oplus_p
\sigma)\rd\sigma$ satisfies the same continuity properties as
$K(x\oplus_p \sigma,y\oplus_p \sigma)$ for fixed $\sigma$.

Applying Lemma~\ref{lemabsconv} to $K_{\sh}^{(1)}$ yields that
$K_{\sh}^{(1)}$ can be represented by an absolutely convergent
$B_{p}^{(1)}$-series.

Finally, note that $K_{\sh}$ is the product of one-dimensional
kernels of the form $K_{\sh}^{(1)}$. This yields the result.
\end{proof}

Now we can give the proof of Proposition~\ref{fo_shinv_kernel_a}.

\begin{proof}
The proof follows the lines of the proof of \cite[Proposition~12.5]{DP10}.

Note that $K_{s,\bsgamma} \in L_2([0,1]^{2s})$. As
$K_{s,\bsgamma}$ is of product form,
$K_{s,\bsgamma}(\bsx,\bsy)=\prod_{j=1}^s K_{j,\gamma_j}(x_j,y_j)$,
we only need to find the $p_j$-adic shift invariant kernels
$K_{{\sh},j}$ associated with $K_{j,\gamma_j}$. The $\bsp$-adic
shift invariant kernel in dimension $s>1$ is then just the product
of the $K_{\sh,j}$.

We have to compute (the indices $j$ for the dimension are omitted)
$$\widehat{K}_{\gamma}(k,k)=\int_0^1 \int_0^1 (1+\gamma
\min(1-x,1-y)) \overline{\beta_k(x)} \beta_k(y) \rd x \rd y.$$

First, it is easy to show that
$$\widehat{K}_{\gamma}(0,0)=1+\frac{\gamma}{3}.$$

Now we turn to the case $k \not=0$. Note that
$\min(1-x,1-y)=1-\max(x,y)=1-\tfrac{1}{2}(x+y+|x-y|)$. Since for
$k \not=0$ we have $\int_0^1 \beta_k(x)\rd x=0$ and  $\int_0^1
\int_0^1 \overline{\beta_k(x)}\beta_k(y) \rd x \rd y=0$ it follows
that
\begin{eqnarray}\label{khat1}
\widehat{K}_{\gamma}(k,k) & = &  -\frac{\gamma}{2} \int_0^1 \int_0^1 (x+y+|x-y|)\ \overline{\beta_k(x)} \beta_k(y) \rd x \rd y\nonumber \\
& = & -\frac{\gamma}{2} \int_0^1 \int_0^1 |x-y|\
\overline{\beta_k(x)} \beta_k(y) \rd x \rd y.
\end{eqnarray}

Let again ${\rm e}(x)={\rm e}^{2 \pi \icomp x}$. Let $k = \kappa_{a-1}p^{a-1}
+ \cdots + \kappa_1 p + \kappa_0$, where $a$ is such that
$\kappa_{a-1} \neq 0$. Note that for $x\in [0,1)$ of the form
$x=\frac{x_1}{p}+\frac{x_2}{p^2}+\cdots+\frac{x_a}{p^a}+\cdots$ we
have
\begin{eqnarray*}
\beta_k(x) & = & {\rm e}\left(\left(\frac{\kappa_0}{p}+\frac{\kappa_1}{p^2}+\cdots +\frac{\kappa_{a-1}}{p^a}\right)
(x_1+x_2 p+\cdots +x_a p^{a-1}+\cdots)\right)\\
& = & {\rm e}\left(\frac{\kappa_0 x_1+\kappa_1 x_2+\cdots+\kappa_{a-1}
x_a}{p}+\frac{\kappa_1 x_1+\kappa_2 x_2+\cdots +\kappa_{a-1}
x_{a-1}}{p^2}+\cdots +\frac{\kappa_{a-1} x_1}{p^a}\right).
\end{eqnarray*}

For $u,v\in \{0,1,\ldots,p^a-1\}$ let $u = u_{1} p^{a-1} + \cdots
+ u_{a-1} p + u_a$ and $v = v_1 p^{a-1} + \cdots + v_{a-1} p +
v_a$ be the $p$-adic expansions of $u$ and $v$, respectively. Then
\begin{eqnarray*}
\tau_p(k) &:=&  \int_0^1 \int_0^1 |x-y| \ \overline{\beta_k(x)}  \beta_k(y) \rd x \rd y \\
&=&  \sum_{u=0}^{p^a-1} \sum_{v=0}^{p^a-1} {\rm e}\left(-\left(\frac{\kappa_0 u_1+\cdots+
\kappa_{a-1} u_a}{p}+\cdots +\frac{\kappa_{a-1} u_1}{p^a}\right)\right)\\
&&\hspace{1cm} \mbox{}\times {\rm e}\left(\frac{\kappa_0 v_1+\cdots+\kappa_{a-1} v_a}{p}+\cdots +\frac{\kappa_{a-1} v_1}{p^a}\right)\\
&&\hspace{1cm} \mbox{}\times \int_{u/p^a}^{(u+1)/p^a}
\int_{v/p^a}^{(v+1)/p^a} |x-y| \rd x \rd y.
\end{eqnarray*}
We have the following equalities:
\[
    \int_{u/p^a}^{(u+1)/p^a} \int_{u/p^a}^{(u+1)/p^a} |x-y| \rd x \rd y = \frac{1}{3 p^{3a}}
\]
and for $u\neq v$, we have
\[
   \int_{u/p^a}^{(u+1)/p^a} \int_{v/p^a}^{(v+1)/p^a} |x-y| \rd x \rd y = \frac{|u-v|}{p^{3a}}.
\]
Thus
\begin{eqnarray}\label{tau_sum}
\tau_p(k)
&=& \frac{1}{3 p^{2a}\nonumber}\\
&& + \sum_{\satop{u,v=0}{u\neq v}}^{p^a-1}
{\rm e}\left(\frac{\kappa_0 (v_1-u_1)+\cdots+\kappa_{a-1} (v_a-u_a)}{p}+\cdots +\frac{\kappa_{a-1} (v_1-u_1)}{p^a}\right)
\frac{|u-v|}{p^{3a}}\nonumber \\
&=&  \frac{1}{3 p^{2a}}\nonumber\\
&& + \sum_{u=0}^{p^a-2} \sum_{v=u+1}^{p^a-1} \frac{v-u}{p^{3 a}}
{\rm e}\left(\frac{\kappa_0 (v_1-u_1)+\cdots+\kappa_{a-1} (v_a-u_a)}{p}+\cdots +\frac{\kappa_{a-1} (v_1-u_1)}{p^a}\right)\nonumber\\
&& + \sum_{v=0}^{p^a-2} \sum_{u=v+1}^{p^a-1} \frac{u-v}{p^{3 a}}
{\rm e}\left(\frac{\kappa_0 (v_1-u_1)+\cdots+\kappa_{a-1} (v_a-u_a)}{p}+\cdots +\frac{\kappa_{a-1} (v_1-u_1)}{p^a}\right)\nonumber\\
&=&  \frac{1}{3 p^{2a}}+ \frac{2}{p^{3a}} {\rm
Re}\left[\sum_{u=0}^{p^a-2} \sum_{v=u+1}^{p^a-1} \theta(u,v)
\right],
\end{eqnarray}
where ${\rm Re}[z]$ denotes the real part of a complex number $z$
and where
\begin{eqnarray*}
\theta(u,v) & := & (v-u)  {\rm e}\left(\frac{\kappa_0 (v_1-u_1)+\cdots+\kappa_{a-1} (v_a-u_a)}{p}+\cdots +\frac{\kappa_{a-1} (v_1-u_1)}{p^a}\right)\\
& = & (v-u)  {\rm e}\bigg((v_1-u_1)\left(\frac{\kappa_0}{p}+\frac{\kappa_1}{p^2}+\cdots+\frac{\kappa_{a-1}}{p^a}\right)\\
&&\hspace{0.5cm}
+(v_2-u_2)\left(\frac{\kappa_1}{p}+\frac{\kappa_2}{p^2}+\cdots+\frac{\kappa_{a-1}}{p^{a-1}}\right)+\cdots
+(v_a-u_a) \frac{\kappa_{a-1}}{p} \bigg).
\end{eqnarray*}
Assume now that $u' = u_1 p^{a-1}+ \cdots + u_{a-1} p$ and let $v' =
v_1 p^{a-1} + \cdots + v_{a-1} p$, where $v > u$. Observe that $u'$
and $v'$ are divisible by $p$. We have
\begin{equation}\label{tau_sum_0}
\left|\sum_{u_a=0}^{p-1} \sum_{v_a=0}^{p-1}
\theta(u'+u_a,v'+v_a)\right| =  \left|\sum_{u_a=0}^{p-1}
\sum_{v_a=0}^{p-1} (v_a-u_a) {\rm e}\left(\frac{\kappa_{a-1}
(v_a-u_a)}{p}\right)\right|,
\end{equation}
since $|{\rm e}(x)|=1$ and since $$\sum_{u_a=0}^{p-1} \sum_{v_a=0}^{p-1} (v-u)
{\rm e}\left(\frac{\kappa_{a-1} (v_a-u_a)}{p}\right) =0.$$ Now it
follows in the same way as in \cite[p.~372, Eq.~(12.12)]{DP10}
that the sum \eqref{tau_sum_0} is $0$.

Therefore most terms in the double sum in \eqref{tau_sum} cancel
out. We are left with the following terms: $\theta(u+u_a,u+v_a)$
for $u = 0,\ldots,p^a-p$, where $p | u$, and $0 \leq u_a < v_a
\leq p-1$. We have
\begin{eqnarray*}
\theta(u+u_a,u+v_a)= (u+v_a-u-u_a) {\rm e}\left(\frac{\kappa_{a-1}
(v_a-u_a)}{p}\right)= \theta(u_a,v_a).
\end{eqnarray*}
The sum over all remaining $\theta(u_a,v_a)$ can be calculated
using geometric series. By doing so we obtain
\begin{equation*}
\sum_{u_a=0}^{p-2} \sum_{v_a=u_a+1}^{p-1} \theta(u_a,v_a) =
-\frac{p}{2\sin^2(\kappa_{a-1}\pi/p)}.
\end{equation*}
Inserting in \eqref{tau_sum} gives
\begin{equation}\label{tau_value}
\tau_p(k)=  \frac{1}{3p^{2a}} - \frac{2}{p^{3a}}
\frac{p^a}{2\sin^2(\kappa_{a-1}\pi/p)} =  \frac{1}{p^{2a}}
\left(\frac{1}{3} - \frac{1}{\sin^2(\kappa_{a-1}\pi/p)} \right).
\end{equation}
In view of \eqref{khat1} we obtain
$$\widehat{K}_{\gamma}(k,k) =\frac{\gamma}{2 p^{2a}} \left(\frac{1}{\sin^2(\kappa_{a-1}\pi/p)}-\frac{1}{3}\right).$$

Now set $$r_{p,\gamma}(k):=\widehat{K}_{\gamma}(k,k) .$$ It can be
easily checked that $\sum_{k=0}^\infty r_{p,\gamma}(k) < \infty$.
Further, $K_{\gamma}$ is also continuous. Hence the result follows
from Lemma~\ref{lemshikernel}.
\end{proof}

\section*{Acknowledgements}

The authors would like to thank G.~Leobacher and I.~Pirsic for
comments and discussions.

\begin{small}
\noindent\textbf{Authors' addresses:}\\

\medskip

\noindent Peter Hellekalek\\
Fachbereich Mathematik, Universit\"{a}t Salzburg\\
Hellbrunnerstr.~34, 5020 Salzburg, Austria\\
E-mail: \texttt{peter.hellekalek@sbg.ac.at}.

\medskip

\noindent Peter Kritzer, Friedrich Pillichshammer\\
Institut f\"{u}r Finanzmathematik, Johannes Kepler Universit\"{a}t Linz\\
Altenbergerstr.~69, 4040 Linz, Austria\\
E-mail: \texttt{peter.kritzer@jku.at},
\texttt{friedrich.pillichshammer@jku.at}.
\end{small}

\end{document}